\documentclass[11pt]{amsart}
\usepackage{amsthm,amssymb,url,hyperref}
\usepackage{fullpage}
\usepackage[dvipsnames]{xcolor}
\usepackage{multirow}
\usepackage{tabulary}
\usepackage{booktabs}
\usepackage{changepage}
\usepackage{caption}
\usepackage{comment}
\usepackage{cleveref}

\theoremstyle{plain}
\newtheorem{theorem}{Theorem}[section]

\newtheorem{corollary}[theorem]{Corollary}

\newtheorem{proposition}[theorem]{Proposition}

\theoremstyle{definition}

\def\img#1{\mathrm{Im}(#1)}
\def\aut#1{{\mathrm{Aut}(#1)}}
\def\Z{\mathbb Z}

\def\Q{\mathrm{Aff}}

\def\m#1#2#3#4{\left(\begin{smallmatrix} #1 & #2 \\ #3 & #4\end{smallmatrix}\right)}
\def\mm#1#2#3#4{\left(\begin{matrix} #1 & #2 \\ #3 & #4\end{matrix}\right)}

\newcommand{\zobr}{1-\varphi-\psi}
\newcommand{\zobrs}{1-2\varphi}
\newcommand{\Gfg}{G_{\varphi, \psi}}
\renewcommand{\vec}[1]{\mathbf{#1}}
\newcolumntype{L}[1]{>{\raggedright\let\newline\\\arraybackslash\hspace{0pt}}m{#1}}
\newcolumntype{C}[1]{>{\centering\let\newline\\\arraybackslash\hspace{0pt}}m{#1}}
\newcolumntype{R}[1]{>{\raggedleft\let\newline\\\arraybackslash\hspace{0pt}}m{#1}}

\title{Paramedial quasigroups of prime and prime square order}

\author{\v Zaneta Semani\v sinov\'a}

\address{Department of Algebra, Faculty of Mathematics and Physics, Charles University, Prague, Czech Republic}
\email{zaneta.semanisinova@gmail.com}
\begin{document}

\thanks{Research partly supported by the project SVV-2020-260589 and by an institutional SFG grant}


%\date{\today}

\begin{abstract}
We prove that, for every odd prime number $p$, there are $2p-1$ paramedial quasigroups of order $p$ and
%chybne cislo
%\textcolor{blue}{\frac{11}{2}p^2+\frac{3}{2}p-4}$ 
%spravne cislo:
$6p^2-p-1$
paramedial quasigroups of order $p^2$, up to isomorphism. We present a complete list of those which are simple.
\end{abstract}

\maketitle

\section{Introduction}\label{sec:intro} 

A binary operation $*$ on a set $A$ is called \emph{paramedial}, if it satisfies the identity
\[ (x*y)*(u*v)=(v*y)*(u*x).  \]
For example, subtraction in abelian groups is paramedial. A binary algebraic structure $(A,*)$ with a paramedial operation will be called a \emph{paramedial groupoid}. The structure will be called \emph{paramedial quasigroup}, if the operation $*$ is uniquely divisible from both sides (and thus forms a~latin square).
It is unknown to us where the paramedial law originated. It was considered, for example, in \cite{NK1} under the nickname $\beta_2$, for purely syntactic reasons, being, in a sense, dual to the more famous \emph{medial} law $(x*y)*(u*v)=(x*u)*(y*v)$. A systematic study of paramedial groupoids was initiated by Cho, Je\v zek and Kepka in \cite{CJK1}. The identity also appears in the context of Abel-Grassmann groupoids \cite{Sta-AG}.

In \cite{CJK2}, Cho, Je\v zek and Kepka initiated the classification of simple paramedial groupoids. They proved that finite simple paramedial groupoids come in four mutually exclusive types: quasigroups, zeropotent groupoids, commutative non-zeropotent non-quasigroups, and certain isotopes of rectangular bands. For the latter three types, a complete classification is available \cite{CJK2,CK1,CK2}. However, the only paper addressing the quasigroup case is \cite{SP} (exploiting the observation that the equivalence defined by $x\sim y\Leftrightarrow x*x=y*y$ is a congruence), but it does not contain any classification results.

Our goal is to enumerate paramedial quasigroups of small order and determine those which are simple.
The fundamental tool to study paramedial quasigroups (and many other classes of quasigroups) is their affine representation over abelian groups \cite[Theorem 9]{NK1}. Therefore, using the methods of \cite{Dra}, the enumeration reduces to an analysis of square roots and conjugacy classes in (certain subgroups of) automorphism groups of abelian groups; in the finite case, one can restrict to groups of prime power order.
In the present paper, we focus on cyclic groups and on elementary abelian groups of dimension 2. 

A similar approach was taken to enumerate medial quasigroups of order $p$ \cite{Kir,SV}, $p^2$ \cite{Sta-p2}, $p^3$ and $p^4$ in the idempotent case \cite{Hou}, and $64\neq p^k<128$ \cite{SV} ($p$ prime). One of our motivations was, whether it is feasible to use the same methods in the paramedial setting. As it turns out, it is, however, the analysis is much more complicated, and in one case, we rely on a non-trivial result from number theory.

Let $\mathrm{pq}(n)$ denote the number of paramedial quasigroups of order $n$ and $\mathrm{pq}(G)$ the number of paramedial quasigroups that admit an affine form over a group $G$, up to isomorphism (see Section 2 for the definition of an affine form). The main result of the present paper can be stated as follows.

\begin{theorem} \label{t:main} 
For every odd prime $p$,
\begin{align*} 
\mathrm{pq}(\Z_{p^k})&=2p^k-p^{k-1}+\sum_{i=0}^{k-2}{p^i}, \\
%chybne cislo
%\mathrm{pq}(\Z_p^2)&=\frac{7}{2}p^2+\frac{5}{2}p-5.
%spravne cislo:
\mathrm{pq}(\Z_p^2)&=4p^2-2.
\end{align*}
Moreover, $\mathrm{pq}(\Z_{2^k})=2^{k+1}$ for $k>2$, $\mathrm{pq}(\Z_4)=4$, $\mathrm{pq}(\Z_2)=1$ and $\mathrm{pq}(\Z_2^2)=7$.
\end{theorem}

\begin{corollary}
For every odd prime $p$,
\begin{align*} 
\mathrm{pq}(p)&=2p-1, \\ 
%chybne cislo
%\mathrm{pq}(p^2)&=\frac{11}{2}p^2+\frac{3}{2}p-4. 
%spravne cislo:
\mathrm{pq}(p^2)&=6p^2-p-1.
\end{align*}
\end{corollary}

%su to len tych nad Z_p^2
The affine forms of the quasigroups over $\Z_p^2$ can be found in Table \ref{tab:main}, except for one case, in which the quasigroups were enumerated using Burnside's lemma without finding them explicitly. 
%je to lepsie v sekcii 2 alebo 4?
The proof of Theorem \ref{t:main} occupies most of the paper. 
In Section \ref{s2} we describe the enumeration method. In Section \ref{s3} we apply the method to cyclic groups and in Section \ref{s4} to elementary abelian groups of dimension 2. 
%\textcolor{red}{state an auxilliary result from number theory} to be used in the proof

All quasigroups of prime order are simple. Simple paramedial quasigroups of prime square order will be classified in Section \ref{s5}, the results are summarized in Table \ref{tab:simple}.

\section{Enumeration method}\label{s2}

Given an abelian group $G=(G,+)$, automorphisms $\varphi,\psi$ of $G$ and an element $c\in G$, define a~new operation $*$ on the set $G$ by
\begin{displaymath}
    x*y = \varphi(x)+\psi(y)+c.
\end{displaymath}
The resulting quasigroup $(G,*)$ is said to be \emph{affine over the group} $G$ and it will be denoted by $\Q(G,+,\varphi,\psi,c)$; the quintuple $(G,+,\varphi,\psi,c)$ is called an \emph{affine form} of $(G,*)$.

\begin{theorem}[{N\v emec and Kepka \cite[Theorem 9]{NK1}}]
A quasigroup $(Q,*)$ is paramedial if and only if there is an abelian group $G=(G,+)$, a pair of automorphisms $\varphi,\psi$ of $G$ satisfying $\varphi^2=\psi^2$, and $c\in G$ such that $Q=\Q(G,+,\varphi,\psi,c)$.
\end{theorem}

%upravila som poradie tych viet, tu poslednu vetu treba upravit ale podla mna dava zmysel az na konci
%netreba podrobnejsie vysvetlenie alebo aspon odkaz na clanok, kde je to podrobnejsie?
It follows from the classification of finite abelian groups that $\mathrm{pq}(G\times H)=\mathrm{pq}(G)\cdot \mathrm{pq}(H)$ whenever $G,H$ are abelian groups of coprime order (see \cite{HR} for a gentle introduction into automorphisms of finite abelian groups). Since isotopic groups are isomorphic \cite[Proposition 1.4]{Smi}, a paramedial quasigroup cannot admit affine forms over two non-isomorphic groups, and thus $\mathrm{pq}(n)=\sum \mathrm{pq}(G)$ where the sum runs over all isomorphism representatives of abelian groups of order $n$. These two observations imply that the function $\mathrm{pq}(n)$ is multiplicative.

We will follow the enumeration procedure described in detail in \cite{SV}. It is based on the following theorem, which is a special case of {{\cite[Theorem 3.2]{Dra}} for paramedial quasigroups. 

\begin{theorem}\label{t:alg}
Let $G$ be an abelian group. The isomorphism classes of paramedial quasigroups affine over $G$ are in one-to-one correspondence with the elements of the set
\begin{displaymath}
    \{(\varphi,\psi,c):\varphi\in X,\,\psi\in Y_\varphi,\,c\in G_{\varphi,\psi}\},
\end{displaymath}
where 
\begin{itemize}
	\item $X$ is a set of conjugacy class representatives of the group $\aut{G}$,
	\item $Y_\varphi$ is a set of orbit representatives of the conjugation action of the centralizer $C_\aut{G}(\varphi)$ on the set $S_\varphi=\{\psi\in\aut G:\psi^2=\varphi^2\}$,
	\item $G_{\varphi,\psi}$ is a set of orbit representatives of the natural action of $C_\aut{G}(\varphi)\cap C_\aut{G}(\psi)$ on the set $G/\img{1-\varphi-\psi}$.
\end{itemize}
\end{theorem}

Every triple $(\varphi,\psi,c)$ yields a paramedial quasigroup $\Q(G,\varphi,\psi,c)$, hence, an explicit construction of the sets $X,Y_\varphi,G_{\varphi,\psi}$ provides an explicit construction of the quasigroups.

\section{Enumeration over cyclic groups}\label{s3}

In the present section, we apply Theorem \ref{t:alg} on the cyclic group $G=\Z_{p^k}$ for $p$ prime number. We will identify automorphisms with the corresponding invertible elements modulo $p^k$, i.e., $\aut{\Z_{p^k}}=\Z_{p^k}^* =\{a\in \Z_{p^k}:\ p\nmid a\}$. 

\begin{proof}[Proof of Theorem \ref{t:main}, case $\mathrm{pq}(\Z_{p^k})$]
Since the automorphism group is commutative, the conjugation action is trivial and we have to consider every pair $(\varphi,\psi)\in\Z_{p^k}^*\times\Z_{p^k}^*$ such that $\varphi^2=\psi^2$. For each pair, we consider a complete set of orbit representatives $G_{\varphi,\psi}$ of the action of $\Z_{p^k}^*$ on $\Z_{p^k}/\img{1-\varphi-\psi}$. 
%Now, $\img{1-\varphi-\psi}$ is equal to $p^iG$ if and only if $p^i\mid 1-\varphi-\psi$ and $p^{i+1}\nmid 1-\varphi-\psi$.

\textbf{Case $p$ odd}. The equality $\varphi^2=\psi^2$ in $\Z_{p^k}$ is equivalent to $(\varphi+\psi)(\varphi-\psi)=0$, therefore the only possible cases are:

\begin{enumerate}
\item $\psi=-\varphi$
\item $\psi=\varphi$
\item $p\mid\varphi+\psi$ and $p\mid\varphi-\psi$
\end{enumerate}
The last case implies that $p\mid 2\varphi$. Since $p$ is odd, we have a contradiction with $\varphi\in\Z_{p^k}^*.$ We will consider the first two cases.

Case $\psi=-\varphi$. The group $\Z_{p^k}/\img{\zobr}$ is trivial, hence we choose $\Gfg=\{0\}$ and the number of possible triples $(\varphi, \psi, c)$ is
\[p^k-p^{k-1}.\]

Case $\psi=\varphi$. Then $\img{1-2\varphi}=p^i\Z_{p^k}$, where $p^i=\mathrm{gcd}(1-2\varphi, p^k)$.

\textit{Subcase $i=0$:}\\
$\img{\zobrs}=\Z_{p^k}$ and $p\nmid1-2\varphi$, hence $\varphi\not\equiv2^{-1}(\textup{mod}\;p)$. The group $\Z_{p^k}/\img{\zobrs}$ is trivial and we choose $\Gfg=\{0\}$. The number of possible automorphisms $\varphi$ and also the number of corresponding triples $(\varphi, \psi, c)$ is
\[p^k-2p^{k-1}.\]

\textit{Subcase $i\in\{1, \dots, k-1\}$} (only for $k>1$):\\
$\img{\zobrs}$ is equal to $p^i\Z_{p^k}$ if and only if $p^i\mid \zobrs$ and $p^{i+1}\nmid \zobrs$, equivalently
$\varphi\equiv2^{-1}(\textup{mod}\;p^i)$ and 
$\varphi\not\equiv2^{-1}(\textup{mod}\;p^{i+1})$. The elements of $\Z_{p^k}^*$ act on the representants of cosets in $\Z_{p^k}/\img{\zobrs}$ by multiplication, therefore we can choose the set $\Gfg$ consisting of zero and the powers of $p$, in particular $\Gfg=\{0, p^0, \dots,  p^{i-1}\}$. The number of possible triples $(\varphi, \psi, c)$ in this case is \[(p^{k-i}-p^{k-i-1})(i+1).\]

\textit{Subcase $i=k$:} \\
$\img{\zobrs}=\{0\}$, which occurs if and only if $\varphi=2^{-1}$ in $\Z_{p^k}^*$. Similarly to the previous case, we can choose $\Gfg=\{0,\,p^0, \dots,\, p^{k-1}\}$. The number of possibilities in this case is equal to the number of possible constants $c$, which is
\[k+1.\]

Summing up the contributions to $\mathrm{pq}(\Z_{p^k})$ in all cases, we get 
\[(p^k-p^{k-1})+(p^k-2p^{k-1})+\sum_{i=1}^{k-1}{(p^{k-i}-p^{k-i-1})(i+1)}+(k+1),\]
which, after rearrangement, gives
\[2p^k-3p^{k-1}+(2p^{k-1}+p^{k-2}+p^{k-3}+\dots+p+1)=2p^k-p^{k-1}+\sum_{i=0}^{k-2}{p^i},\]
which is the value of $\mathrm{pq}(\Z_{p^k})$.

\textbf{Case $p=2$.} The special cases with $k=1, 2$ are straightforward and left to the reader. We count the number $\mathrm{pq}(\Z_{2^k})$ for $k>2$. The automorphism group decomposes as $\aut{\Z_{2^k}}=\Z_{2^k}^*\simeq \Z_2\times\Z_{2^{k-2}}$, let us denote the isomorphism $f:\Z_{2^k}^*\rightarrow\Z_2\times\Z_{2^{k-2}}$. Then every pair of automorphisms $(\varphi, \psi)$ corresponds to the pair $(f(\varphi), f(\psi))$, where $f(\varphi)=(a_1, a_2)$, $f(\psi)=(b_1, b_2)$, $a_1, b_1\in\Z_2$ and $a_2, b_2 \in \Z_{2^{k-2}}$. Reformulating the condition $\varphi^2=\psi^2$ in terms of this isomorphism, we get the equality $2(a_1,  a_2)=2(b_1, b_2)$, which holds if and only if
\[a_2 \equiv b_2\; (\textup{mod}\;2^{k-3}).\]

For a fixed element $\varphi\in \textup{Aut}(\Z_{2^k})$, there are precisely two elements $b_2\in\Z_{2^{k-2}}$ satisfying the stated condition. Since the element $b_1\in\Z_2$ can be chosen arbitrarily, there are four possible choices of the automorphism $\psi\in\aut{\Z_{2^k}}.$ 

It remains to determine the set $\Gfg$. Since $2\nmid 1-\varphi-\psi$, we get Im$(1-\varphi-\psi)=\Z_{2^k}$, hence $\Z_{2^k}/\textup{Im}(\zobr)$ is a trivial group and we can choose $\Gfg=\{0\}$. Therefore, the number of possible triples $(\varphi, \psi, c)$ is  \[(2^k-2^{k-1})\cdot 4=2^{k+1}.\]

\end{proof}

\section{Enumeration over elementary abelian groups of dimension 2}\label{s4}

\subsection{An auxilliary number-theoretic result}\label{s4:solutions}
In this subsection we state a number-theoretic result about the number of solutions of a particular equation in the field $\Z_p$, which will be used in Subsection \ref{s4:enum}. The result is obtained by methods from \cite[Section 2]{conicsmodp}.

Let $p$ be an odd prime number and $a\in \Z_p$ a non-square modulo $p$ (in particular, $a\neq0$).
We want to determine the number of solutions of the equation
\begin{equation*}
q(k,l)=k^2-al^2+(1-2a)l-a=0.
\end{equation*}
We will show that the number of pairs $(k, l)\in \Z_p^2$ satisfying the equation is $p+1$.

Let us denote $\vec{x}=(k, l)^T$ and $M=\left( \begin{smallmatrix}
    1& 0  \\
    0 & -a \\
    \end{smallmatrix} \right)$. Then we can restate the equation as follows:
    \[q(\vec{x})=\vec{x}^TM\vec{x}+(0,\quad1-2a)\vec{x}-a=0.\]
For any vector $\vec{s}$, the equations $q(\vec{x})=0$ and $q'(\vec{x})=q(\vec{x}+\vec{s})=0$ have the same number of solutions. Let $\vec{s}=-\frac{1}{2}M^{-1}\left( \begin{smallmatrix}
    0\\
    1-2a \\
    \end{smallmatrix} \right)$. Simplifying the equation, we obtain (we use symmetry of the matrices $M$ and $M^{-1}$)
\begin{align*}
    q'(\vec{x})&=(\vec{x}+\vec{s})^TM(\vec{x}+\vec{s})+(0,\quad 1-2a)(\vec{x}+\vec{s})-a\\
    &=\vec{x}^TM\vec{x}+(\vec{s}^TM\vec{s}+(0,\quad1-2a)\vec{s}-a)\\
    &=k^2-al^2+(2^{-2}a^{-1}-1).
\end{align*}
The number of solutions $(k, l)$ follows from the following theorem ($2^{-2}a^{-1}-1\neq0$, since $a$ is a~non-square). 

\begin{theorem}[\protect{\cite[Theorem 2.9]{conicsmodp}}]\label{pcries}
Let $p$ be an odd prime and $g(x,y)=rx^2+sy^2+t$ be a polynomial over $\Z_p$ satisfying the following:
\begin{itemize}
    \item $t\neq0,$
    \item $-rs$ is a non-square modulo $p$ (in particular $rs\neq0$).
\end{itemize}
Then the number of solutions $(x, y)$ of the equation $g(x,y)=0$ is $p+1$.
\end{theorem}

\subsection{Calculating square roots of matrices}\label{s4:sqrroot}
We will introduce a method for calculating square roots of $2\times2$ matrices over the field $\Z_p$ for odd primes $p$. The method is based on the Cayley-Hamilton theorem and is adopted from \cite{sqrroot}.

Let $A$ be a $2\times2$ matrix over $\Z_p$, our aim is to find all matrices $\sqrt{A}$ satisfying $(\sqrt{A})^2=A$. In this context, the identity matrix will be denoted by $I$.

The Cayley-Hamilton theorem for $2\times2$ matrices asserts that 
\begin{equation}\label{C-Hth}
    M^2-\textup{tr}(M)M+\textup{det}(M)I=0.
\end{equation}
Let us denote $\tau=\textup{tr}(\sqrt{A})$, $\delta=\textup{det}(\sqrt{A})$, $T=\textup{tr}(A)$, $\Delta=\textup{det}(A)$. Clearly $\delta^2=\Delta$.
Setting $M$=$\sqrt{A}$ in (\ref{C-Hth}) yields  
\begin{equation}\label{CH-sqrt}
    A=\tau \sqrt{A}-\delta I.
\end{equation}

If we assume $\tau\neq0$, we obtain
\begin{equation}\label{sqrt}
    \sqrt{A}=\frac{1}{\tau}(A+\delta I).
\end{equation}

In order to express $\tau$, we use the Cayley-Hamilton theorem for $M=A$, which gives 
\begin{equation}\label{CH-normal}
    TA-\Delta I=A^2\overset{\text{(\ref{CH-sqrt})}}=(\tau\sqrt{A}-\delta I)^2=\tau^2A-2\tau\delta\sqrt{A}+\delta^2 I
    \overset{\text{(\ref{sqrt})}}=
       \tau^2A-2\delta(A+\delta I)+\delta^2 I=\tau^2A-2\delta A-\Delta I,
\end{equation}

and thus $TA=(\tau^2-2\delta)A$, which implies $T=\tau^2-2\delta$. For every square $x\in\Z_p$, we choose an element $y$ such that $y^2=x$ and denote it by $\sqrt{x}$. If the corresponding square roots in $\Z_p$ exist, we can express the square roots of $A$ in the form
\begin{equation}\label{sqrt-formula}
    \sqrt{A}=\frac{\pm 1}{\sqrt{T+2\delta}}(A+\delta I),
\end{equation}
where $\delta=\pm \sqrt{\Delta}$. Therefore, we obtain at most four square roots of $A$ with $\tau\neq0$.

\textit{Case $A\neq cI, c\in\Z_p$:} Since $A$ is not a multiple of the identity matrix, it follows from (\ref{CH-sqrt}) that $\tau\neq0$. Therefore we obtain all square roots of $A$ from (\ref{sqrt-formula}).

\textit{Case $A=cI$ for $c\in\Z_p$:}
We will search for all matrices $\sqrt{A}$ in the following form:
\begin{equation*}
\sqrt{A}=\begin{pmatrix}
k & l \\
m & n\\
\end{pmatrix}, \quad
\begin{pmatrix}
c & 0 \\
0 & c\\
\end{pmatrix}=
\begin{pmatrix}
k & l \\
m & n\\
\end{pmatrix}^2 =
\begin{pmatrix}
k^2+lm & l(k+n) \\
m(k+n) & n^2+lm\\
\end{pmatrix}
\end{equation*}

If $\tau\neq0$, then it follows from (\ref{CH-sqrt}) that $\delta \neq -c$, hence $\delta=c$. Then if $c$ is a square, by (\ref{sqrt-formula}) we obtain square roots in the form
\begin{equation*}
    \frac{\pm 1}{\sqrt{2c+2c}}(cI+cI)=
    \pm \begin{pmatrix}
    \sqrt{c} & 0 \\
     0 & \sqrt{c} \\
    \end{pmatrix} 
\end{equation*}
Otherwise, $\tau=0$, i.e., $n=-k$, so we obtain square roots in the form
\begin{equation*}
    \begin{pmatrix}
    k& l \\
    m & -k\\
    \end{pmatrix},\; k, l, m\in\Z_p,\; k^2+lm=c.
\end{equation*}

\begin{comment}
We obtain four equations for the elements of the matrix including $m(k+n)=0$ a $l(k+n)=0$. Therefore, one of the following cases occurs:
\begin{itemize}
    \item $\mathbf{m=0}$, $\mathbf{l=0}$: The remaining two equations imply $k^2=n^2=c$. If $c=(\sqrt{c})^2$ for a fixed element $\sqrt{c}$, then we obtain 4 possibilities for $\sqrt{A}$ in this case, depending on the choice $k=\pm \sqrt{c}$ and $n=\pm \sqrt{c}$ (otherwise, this case does not provide any solution).
    
    \item $\mathbf{k=-n}$, $\mathbf{l\neq0}$: In this case the equation holds if and only if $m=l^{-1}(c-k^2)$, which gives us square roots in the form 
    \begin{equation*}
    \begin{pmatrix}
    k& l \\
    l^{-1}(c-k^2) & -k\\
    \end{pmatrix},\; k, l\in\mathbb{Z}_p,\;l\neq0
    \end{equation*}
    
     \item $\mathbf{k=-n}$, $\mathbf{m\neq0}$: Similarly to the previous case we get square roots in form
    \begin{equation*}
    \begin{pmatrix}
    k& m^{-1}(c-k^2)  \\
    m & -k\\
    \end{pmatrix},\; k, m\in\mathbb{Z}_p,\; m\neq0
    \end{equation*}
    For most choices of $k$, this form gives us square roots, which are already included in the previous case. In particular, the only additional square roots correspond to the case $k^2=c$, which (if $c$ is a square) yields square roots in the form
    \begin{equation*}
    \begin{pmatrix}
    \pm \sqrt{c}& 0  \\
    m & \mp \sqrt{c}\\
    \end{pmatrix}, 
    \end{equation*} where $m\neq0$.
    
\end{itemize}
\end{comment}

\subsection{Enumeration}\label{s4:enum}

Now, we are ready to apply Theorem \ref{t:alg} on the group $G=\Z_p^2$. We will identify automorphisms with their matrices, considering $\aut{G}=GL(2,p)$. 

\begin{table}[ht]
\[
\begin{array}{|l|l|} \hline
\varphi & C(\varphi) \\\hline
\mm a00a,\ a\neq 0 & GL(2,p) \\\hline
\mm a00b,\ 0<a<b & \left\{\mm u00v:\ u,v\neq 0\right\} \\\hline
\mm a10a,\ a\neq 0 & \left\{\mm uv0u:\ u\neq 0\right\} \\\hline
\mm 01ab,\ x^2-bx-a \text{ irreducible} & \left\{\mm uv{av}{u+bv}:\ u\neq0\text{ or }v\neq 0\right\}.\\\hline
\end{array}
\]
\caption{Conjugacy class representatives in $GL(2,p)$ and their centralizer subgroups.}
\label{tab:X}
\end{table}

\begin{proof}[Proof of Theorem \ref{t:main}, case $\mathrm{pq}(\Z_p^2)$]
Let $G=\Z_p^2$ with $p$ odd (the case $p=2$ is left to the reader). 
The set $X$ of conjugacy class representatives in $\aut{G}=GL(2,p)$ can be chosen as in Table \ref{tab:X}. The four types of representatives correspond to the diagonalizable matrices with one eigenvalue, the diagonalizable matrices with two distinct eigenvalues, the non-diagonalizable matrices with an eigenvalue in $\Z_p$, and the non-diagonalizable matrices with eigenvalues in the quadratic extension, respectively. The last case is represented by matrices $\m 01ab$ such that the polynomial $x^2-bx-a$ is irreducible over $\Z_p$. 

The centralizer subgroups are also displayed in Table \ref{tab:X} (here and later on, we will omit the index in the centralizer notation).

For every matrix $\varphi\in X$ we will do the following: We determine the set $S_\varphi$, using the method from Subsection \ref{s4:sqrroot} for calculating the square roots $\psi$ of the matrices $\varphi^2$. Then we will choose the orbit representatives of the conjugation action of the centralizer $C(\varphi)$ on the set $S_\varphi$ and form the set $Y_\varphi$. Finally, we will determine the set $\Gfg$ and calculate the number of triples $(\varphi,\,\psi,\,c):\:\varphi\in X,$ $\psi \in Y_\varphi,$ $c\in\Gfg$. The results are summarized in Table \ref{tab:main}.

The size of the set $\Gfg$ will be determined by the following procedure: If $\zobr$ is a regular matrix, then $G/\img{\zobr}$ is a trivial group and we can choose $\Gfg=\{\vec{0}\}$. If the rank of the matrix $\zobr$ is one, then $G/\img{\zobr}\simeq\Z_p$ and since all of the centralizer subgroups contain all matrices $\m u00u, u\neq 0$, we can take $\Gfg=\{\vec{0}, \vec{w}\}$, where $[\vec{w}]$ is any non-zero element of $G/\img{\zobr}$ (i.e., $\vec{w}\notin\img{\zobr}$). If the rank of the matrix $\zobr$ is zero, then the situation depends on the centralizers $C(\varphi)$, $C(\psi)$ and will be discussed separately in that case.

\textbf{Case $\varphi=\m a00a$}: 

\[S_\varphi=\left\{
    \left( \begin{smallmatrix}
    \pm a& 0  \\
    0 & \pm a\\
    \end{smallmatrix} \right);\,
    \left( \begin{smallmatrix}
    k& l \\
    m & -k\\
    \end{smallmatrix} \right):\, k^2+lm=a^2
    \right\}\]
All matrices in $S_\varphi$ are diagonalizable and the non-diagonal matrices have eigenvalues $a$ and $-a$. The centralizer $C(\varphi)$ is the whole $GL(2, p)$, therefore we can choose the set of representatives of the conjugation action of $C(\varphi)$ on $S_\varphi$ as
\[Y_\varphi= \left\{
    \left( \begin{smallmatrix}
    a& 0  \\
    0 & a\\
    \end{smallmatrix} \right),\;
    \left( \begin{smallmatrix}
    -a& 0  \\
    0 & -a\\
    \end{smallmatrix} \right),\;
    \left( \begin{smallmatrix}
    a& 0  \\
    0 & -a\\
    \end{smallmatrix} \right)
    \right\}.\]
    
Now, we will determine the set $\Gfg$ for every admissible pair $(\varphi, \psi)$:
    
    \begin{itemize}
    
        \item $\psi=\left( \begin{smallmatrix}
    a& 0  \\
    0 & a\\
    \end{smallmatrix} \right)$, 
    $1-\varphi-\psi=\left( \begin{smallmatrix}
    1-2a& 0  \\
    0 & 1-2a\\
    \end{smallmatrix} \right)$

The matrix $\zobr$ is regular if and only if $a\neq2^{-1}$. Then $|\Gfg|=1$ and this case contributes to $\mathrm{pq}(G)$ by $p-2$ possible triples $(\varphi, \psi, c)$. Otherwise, rank$(\zobr)=0$ and necessarily $a=2^{-1}$. Since $C(\varphi)\cap C(\psi)=GL(2, p)$, it is possible to choose $\Gfg=\{\vec{0},\,\vec{w}\}$, where $\vec{w}$ is an arbitrary non-zero vector. Therefore, this case contributes to $\mathrm{pq}(G)$ by $2$ possible triples.
    
        \item $\psi=\left( \begin{smallmatrix}
    -a& 0  \\
    0 & -a\\
    \end{smallmatrix} \right)$, 
    $1-\varphi-\psi=\left( \begin{smallmatrix}
    1& 0  \\
    0 & 1\\
    \end{smallmatrix} \right)$
    
    For any $a$ is rank$(\zobr)=2$, so $|\Gfg|=1$ and the number of possible triples $(\varphi, \psi, c)$ in this case is $p-1$.
    
        \item $\psi=\left( \begin{smallmatrix}
    a& 0  \\
    0 & -a\\
    \end{smallmatrix} \right)$, 
    $1-\varphi-\psi=\left( \begin{smallmatrix}
    1-2a& 0  \\
    0 & 1\\
    \end{smallmatrix} \right)$
    
    Similarly as in the first case, rank$(\zobr)=2$ if and only if $a\neq2^{-1}$ and then there are $p-2$ admissible triples $(\varphi, \psi, c)$. Otherwise, rank$(\zobr)=1$ and we choose $\Gfg=\{\vec{0},\,\vec{w}\}$, $\vec{w}\notin \text{Im}(\zobr)$, and get $2$ other triples.
    
    \end{itemize}
    
Summing up the contributions in this case we get $(p-2)+2+(p-1)+(p-2)+2=3p-1$.

%%%%%%%%%%%%%%%%%%%%%%%%%%%%%%%%%%%%%%%%%%%%

\textbf{Case $\varphi=\m a00b$}: We need to distinguish two subcases.

\textit{Subcase $b\neq-a$}:
    \[S_\varphi=\left\{
    \left( \begin{smallmatrix}
    \pm a& 0  \\
    0 & \pm b\\
    \end{smallmatrix} \right),\;
    \left( \begin{smallmatrix}
    \pm a& 0  \\
    0 & \mp b\\
    \end{smallmatrix} \right)
    \right\}\]

Every $\psi\in S_\varphi$ commutes with the elements of $C(\varphi)$, therefore we choose $Y_\varphi=S_\varphi$.

\textit{Subcase $b=-a$}:
\[S_\varphi=\left\{
    \left( \begin{smallmatrix}
    a& 0  \\
    0 & a\\
    \end{smallmatrix} \right);\,
    \left( \begin{smallmatrix}
    - a& 0  \\
    0 & - a\\
    \end{smallmatrix} \right);\,
    \left( \begin{smallmatrix}
    k& l \\
    m & -k\\
    \end{smallmatrix} \right):\, k^2+lm=a^2
    \right\}\]
    
All diagonal matrices in $S_\varphi$ commute with the elements of $C(\varphi)$, hence they are all elements of $Y_\varphi$. In the non-diagonal matrices, we can change the non-zero non-diagonal elements by conjugation to any non-zero value, therefore we include a representative for matrices with $l=0$ and representatives for matrices with $l\neq0$. Hence we can choose the set

\[Y_\varphi=\left\{
    \left( \begin{smallmatrix}
    \pm a& 0  \\
    0 & \pm a\\
    \end{smallmatrix} \right);\;
    \left( \begin{smallmatrix}
    \pm a& 0  \\
    0 & \mp a\\
    \end{smallmatrix} \right);\;
    \left( \begin{smallmatrix}
    \pm a& 0  \\
    1 & \mp a\\
    \end{smallmatrix} \right); \;
    \left( \begin{smallmatrix}
    k& 1 \\
    a^2-k^2 & -k\\
    \end{smallmatrix} \right):\, k\in \Z_p
    \right\}.\]
    
Now, we are ready to discuss the structure of the set $\Gfg$ depending on the choice of automorphisms $\varphi, \psi$. Firstly we count the triples $(\varphi, \psi, c)$ where the matrix $\psi$ is diagonal for both subcases together, independently of the value of $b$.

\begin{itemize}
    
    \item $\psi=\left( \begin{smallmatrix}
    a& 0  \\
    0 & b\\
    \end{smallmatrix} \right)$, 
    $1-\varphi-\psi=\left( \begin{smallmatrix}
    1-2a& 0  \\
    0& 1-2b\\
    \end{smallmatrix} \right)$
    
    The matrix $\zobr$ is regular if and only if $a,\,b\neq2^{-1}$. Under these conditions, $|\Gfg|=1$ and number of admissible triples $(\varphi, \psi, c)$ is $\binom{p-2}{2}$. If this condition is not satisfied, rank$(\zobr)=1$ (because $a<b$) and $a=2^{-1}$ or $b=2^{-1}$. In this case we choose $\Gfg=\{\vec{0},\,\vec{w}\}$, where $\vec{w}\notin\text{Im}(\zobr)$, and number of possible triples is  $2(p-2)$.
    
        \item $\psi=\left( \begin{smallmatrix}
    -a& 0  \\
    0 & -b\\
    \end{smallmatrix} \right)$, 
    $1-\varphi-\psi=\left( \begin{smallmatrix}
    1& 0  \\
    0 & 1\\
    \end{smallmatrix} \right)$
    
    Rank of the matrix $\zobr$ is always two, therefore $|\Gfg|=1$ and this case contributes to $\mathrm{pq}(G)$ by $\binom{p-1}{2}$.
    
    \item $\psi=\left( \begin{smallmatrix}
    -a& 0  \\
    0 & b\\
    \end{smallmatrix} \right)$, 
    $1-\varphi-\psi=\left( \begin{smallmatrix}
    1& 0  \\
    0& 1-2b\\
    \end{smallmatrix} \right)$
    
    If $b\neq2^{-1}$, then rank$(\zobr)=2$ and we have $|\Gfg|=1$. This gives $c_1$ admissible triples $(\varphi, \psi, c)$, where by $c_1$ we denote the number of pairs $(a, b)$ satisfying $0<a<b$ and $b\neq2^{-1}$. If $b=2^{-1}$, rank$(\zobr)=1$ and we choose $\Gfg=\{\vec{0},\,\vec{w}\}$, $\vec{w}\notin\text{Im}(\zobr)$, which yields $2d_1$ triples $(\varphi, \psi, c)$, where $d_1$ denotes the number of possibilities to choose $a$ that satisfies $0<a<2^{-1}$.
    
    \item $\psi=\left( \begin{smallmatrix}
    a& 0  \\
    0 & -b\\
    \end{smallmatrix} \right)$, 
    $1-\varphi-\psi=\left( \begin{smallmatrix}
    1-2a & 0  \\
    0& 1\\
    \end{smallmatrix} \right)$
    
Analogously to the previous case, if $a\neq2^{-1}$, then the number of possible triples is $c_2$, which denotes the number of pairs $(a, b)$ satisfying $0<a<b$ and $a\neq2^{-1}$. Conversely, if $a=2^{-1}$, we get $2d_2$ possible triples, where $d_2$ denotes the number of possible choices of $b$ satisfying $2^{-1}<b$.

    \end{itemize}

%tu je pravdepodobne chyba
%\textcolor{blue}{Now we can calculate the values $c_1+c_2$ and $d_1+d_2$. The value $c_1+c_2=\binom{p-2}{2}+(p-2)$ corresponds to the choice of $0<a<b$ satisfying that $a\neq2^{-1}$ or $b\neq2^{-1}$. The value $d_1+d_2=p-2$ counts the pairs $(a, b)$, where $a=2^{-1}$ or $b=2^{-1}$.}

%odstavec so spravnymi cislami:
Now we can calculate the values $c_1+c_2$ and $d_1+d_2$. The value $c_1+c_2=2\binom{p-2}{2}+(p-2)$ corresponds to the choice of $0<a<b$ satisfying that $a\neq2^{-1}$ or $b\neq2^{-1}$, with the choices where $a\neq2^{-1}$ and $b\neq2^{-1}$ counted twice. The value $d_1+d_2=p-2$ is equal to the number of pairs $(a, b)$, where $a=2^{-1}$ or $b=2^{-1}$.

We can now move to the pairs of automorphisms $(\varphi, \psi)$, which are specific for the subcase $b=-a$, therefore we will consider $\varphi=\m  a00{-a}$.

\begin{itemize}
    \item $\psi=\left( \begin{smallmatrix}
    a& 0  \\
    1 & -a\\
    \end{smallmatrix} \right)$, 
    $1-\varphi-\psi=\left( \begin{smallmatrix}
    1-2a & 0  \\
    -1& 1+2a\\
    \end{smallmatrix} \right)$
    
    Matrix $\zobr$ is regular if and only if $a\neq\pm 2^{-1}$. Then $|\Gfg|=1$ and the number of contributed triples $(\varphi, \psi, c)$ is equal to the number of possible automorphisms $\varphi$, which is $\frac{p-1}{2}-1=\frac{p-3}{2}$, because $0<a<-a$. Otherwise, rank$(\zobr)=1$ and $a=\pm 2^{-1}$. Since $0<a<-a$, there is only one possible choice of $a$. Under these conditions, we choose $\Gfg=\{\vec{0},\,\vec{w}\}$, where $\vec{w}\notin\text{Im}(\zobr)$, and get $2$ possible triples.
    
    \item $\psi=\left( \begin{smallmatrix}
    -a& 0  \\
    1 & a\\
    \end{smallmatrix} \right)$, 
    $1-\varphi-\psi=\left( \begin{smallmatrix}
    1 & 0  \\
    -1& 1  \\
    \end{smallmatrix} \right)$
    
    The matrix $\zobr$ is always regular, therefore $|\Gfg|=1$ and this case contributes to $\mathrm{pq}(G)$ by $\frac{p-1}{2}$, because condition $0<a<-a$ is satisfied by half of the non-zero elements in $\Z_p$. 
    
    \item $\psi=\left( \begin{smallmatrix}
    k& 1  \\
    a^2-k^2 & -k\\
    \end{smallmatrix} \right),\; k\in \mathbb{Z}_p$, 
    $1-\varphi-\psi=\left( \begin{smallmatrix}
    1-a-k & -1  \\
    k^2-a^2 & 1+a+k  \\
    \end{smallmatrix} \right)$
    
Clearly, we have rank$(\zobr)>0$. 
In particular, rank$(\zobr)=2$ if and only if det$(\zobr)\neq0$, equivalently $k\neq2^{-1}a^{-1}-a$. Then $|\Gfg|=1$ and contribution to the value $\mathrm{pq}(G)$ is $\frac{(p-1)^2}{2}$, because $0<a<-a$. Conversely, if det$(\zobr)=0$, then rank$(1-\varphi-\psi)=1$ and $k=2^{-1}a^{-1}-a$. Hence we choose $\Gfg=\{\vec{0},\,\vec{w}\}$, where $\vec{w}~\notin~\text{Im}(\zobr)$, and the number of admissible triples $(\varphi, \psi, c)$ is $2\cdot\frac{p-1}{2}=p-1$.
    
\end{itemize}

%chybny vypocet
%\textcolor{blue}{Using the values of the expressions $c_1+c_2$ and $d_1+d_2$ calculated above, we count the total number of contributed triples $(\varphi, \psi, c)$ in this case: 
%$\binom{p-2}{2}+2(p-2)+\binom{p-1}{2}+c_1+2d_1+c_2+2d_2+ \frac{p-3}{2}+2+\frac{p-1}{2}+\frac{(p-1)^2}{2}+p-1=2p^2-\frac{p}{2}-\frac{7}{2}$.}

%spravny sucet:
Using the values of the expressions $c_1+c_2$ and $d_1+d_2$ calculated above, we count the total number of contributed triples $(\varphi, \psi, c)$ in this case: 
$\binom{p-2}{2}+2(p-2)+\binom{p-1}{2}+c_1+2d_1+c_2+2d_2+ \frac{p-3}{2}+2+\frac{p-1}{2}+\frac{(p-1)^2}{2}+p-1=\frac{5}{2}p^2-3p-\frac{1}{2}$.
%%%%%%%%%%%%%%%%%%%%%%%%%%%%%%%%%%%%%%%%%%%%

\textbf{Case $\varphi=\m a10a$}: Using the method introduced in Subsection \ref{s4:sqrroot}, we obtain
    \[S_\varphi=\left\{ \pm
    \left( \begin{smallmatrix}
    a& 1 \\
    0 & a\\
    \end{smallmatrix} \right)
    \right\}.\]
Since both matrices in $S_\varphi$ commute with the matrices in $C(\varphi)$, we choose $Y_\varphi=S_\varphi$.

Finally, we can choose the sets $\Gfg$ for the pairs of the automorphisms $(\varphi, \psi)$ and find the number of contributed quasigroups.    
    \begin{itemize}
    
        \item $\psi=\left( \begin{smallmatrix}
    a& 1  \\
    0 & a\\
    \end{smallmatrix} \right)$, 
    $1-\varphi-\psi=\left( \begin{smallmatrix}
    1-2a& -2  \\
    0 & 1-2a\\
    \end{smallmatrix} \right)$
    
    We have rank$(\zobr)=2$ if and only if $a\neq2^{-1}$. Then $|\Gfg|=1$ and the number of contributed triples is $p-2$. Otherwise, it holds that rank$(\zobr)=1$ and $a=2^{-1}$, hence we choose $\Gfg=\{\vec{0},\,\vec{w}\}$, where $\vec{w}\notin\text{Im}(\zobr)$, and get $2$ possible triples.
    
        \item $\psi=\left( \begin{smallmatrix}
    -a& -1  \\
    0 & -a\\
    \end{smallmatrix} \right)$, 
    $1-\varphi-\psi=\left( \begin{smallmatrix}
    1& 0  \\
    0 & 1\\
    \end{smallmatrix} \right)$
    
    For every $a$ holds that rank$(\zobr)=2$, therefore there are as many admissible triples $(\varphi, \psi, c)$ as possible values of $a$, which is $p-1$.

    \end{itemize}
  
This case contributes to $\mathrm{pq}(G)$ the total of $(p-2)+2+(p-1)=2p-1$ triples $(\varphi, \psi, c)$.

%%%%%%%%%%%%%%%%%%%%%%%%%%%%%%%%%%%%%%%%%%%%%

\textbf{Case $\varphi=\m 01ab$}: The set $S_\varphi$ clearly contains the matrices $\pm \varphi$. Both matrices commute with the elements of $C(\varphi)$ and therefore have to be elements of the set $Y_\varphi$.

Consider $\psi\in Y_\varphi$. We distinguish the following two cases.

\textit{Subcase $\psi=\pm \varphi$:}
 \begin{itemize}
    
        \item $\psi=\left( \begin{smallmatrix}
    0& 1  \\
    a & b\\
    \end{smallmatrix} \right)$, 
    $1-\varphi-\psi=\left( \begin{smallmatrix}
    1& -2  \\
    -2a & 1-2b\\
    \end{smallmatrix} \right)$
    
    Determinant of the matrix $\zobr$ is $1-2b-4a$, hence it is singular if and only if $4a=1-2b$, equivalently $b^2+4a=(b-1)^2$, which contradicts irreducibility of the polynomial $x^2-bx-a$. Therefore, $|\Gfg|=1$. Since there are $\frac{1}{2}(p^2-p)$ monic irreducible polynomials of degree 2 over $\mathbb{F}_p$, this case yields $\frac{1}{2}(p^2-p)$ triples $(\varphi, \psi, c)$.
    
        \item $\psi=\left( \begin{smallmatrix}
    0& -1  \\
    -a & -b\\
    \end{smallmatrix} \right)$, 
    $1-\varphi-\psi=\left( \begin{smallmatrix}
    1& 0  \\
    0 & 1\\
    \end{smallmatrix} \right)$
    
    The matrix $\zobr$ is regular for all $a, b$, therefore, $|\Gfg|=1$ and similarly to the previous case, this case contributes to $\mathrm{pq}(G)$ by $\frac{1}{2}(p^2-p)$.

    \end{itemize}

\textit{Subcase $\psi\neq\pm \varphi$:}
Using the method from Subsection \ref{s4:sqrroot} and that $b^2+4a$ is a non-square modulo $p$ (because of the irreducibility of $x^2-bx-a$), we observe that if $b\neq 0$, then $Y_\varphi=S_\varphi=\left\{ \pm
    \left( \begin{smallmatrix}
    0& 1  \\
    a & b\\
    \end{smallmatrix} \right)
    \right\}$. Therefore necessarily $b=0$ in this case and $\varphi= \m 01a0$ with $a$ non-square modulo $p$. The set of all square roots of $\varphi^2$ (including $\pm\varphi$) is then
\[S_\varphi=\left\{
    \left( \begin{smallmatrix}
    k& l \\
    l^{-1}(a-k^2) & -k\\
    \end{smallmatrix} \right):\: l\neq0
    \right\}.\]
    
In this case, we are not able to list the orbit representatives of the conjugation action of $C(\varphi)$ on the set $S_\varphi$. Instead, we will use Burniside's lemma to find the number of triples $(\varphi, \psi, c)$. Firstly, we find the number of the conjugacy classes and their cardinalities. Then, independently of the particular conjugacy class representatives, we find the cardinality of the set $\Gfg$ depending on the conjugacy class of $\psi$.

If we set $k=0$, $l=1$ in the matrix $\m kl{l^{-1}(a-k^2)}{-k}$, we get $\m 01a0$. Similarly, the choice $k=0$, $l=-1$ leads to the matrix $\m 0{-1}{-a}0$. These two matrices are equal to $\pm \varphi$ and represent single element conjugacy classes.

Let us denote matrices $C(u, v)=\left( \begin{smallmatrix}
    u& v \\
    av & u\\
    \end{smallmatrix} \right)$ from $C(\varphi)$ and $S(k, l)=\left( \begin{smallmatrix}
    k& l \\
    l^{-1}(a-k^2) & -k\\
    \end{smallmatrix} \right)$ from $S_\varphi$. It is straightforward to find for which values of $u, v, k, l$ the equality $C(u, v)S(k, l)C(u, v)^{-1}=S(k, l)$ holds. If $v=0$, then the equality holds for all admissible values $k, l, u$ and matrix $C(u, 0)$ commutes with all elements in $S_\varphi$. If $v\neq0$, then the equality holds if and only if $k=0$ and $l=\pm1$.
  
This observation has two crucial consequences. Firstly, we can use it to determine the number of orbits of the conjugation action of $C(\varphi)$ on $S_\varphi$. The centralizer $C(\varphi)$ contains $p-1$ matrices of the form $C(u, 0)$, each of them with $|S_\varphi|=p(p-1)$ fixed points. The remaining $p(p-1)$ matrices in $C(\varphi)$ are in the form $C(u, v)$, where $v\neq0$, and have only $2$ fixed points, in particular $S(0, 1)$ and $S(0, -1)$. For a matrix $C\in C(\varphi)$ denote $f_C$ the number of its fixed points of the conjugation action. Then we obtain from Burnside's lemma that the number of orbits of this action is
\[\frac{1}{|C(\varphi)|}\sum_{C\in C(\varphi)}|f_C|=
\frac{1}{p^2-1}(p(p-1)(p-1)+2p(p-1))=p.\]

The second consequence is the sizes of the orbits. We already know that the matrices $S(0, \pm1)$ have single element orbits. Using the observation above, we obtain that the other elements in $S_\varphi$ are stabilized only by $p-1$ diagonal matrices $C(u, 0)$, $u\neq0$. Therefore, the sizes of the remaining $p-2$ orbits are equal to $|C(\varphi)|/(p-1)=p+1$. 
Overall, the conjugation action of $C(\varphi)$ on $S_\varphi$ has two single element orbits and $p-2$ orbits with $p+1$ elements.

Now we are able to find the size of the set $\Gfg$ for the pairs of automorphisms $(\varphi, \psi)$, $\psi\in Y_\varphi\setminus\{\pm \varphi\}$.  %We already know that there are $p-2$ such elements, each of them representing orbit with $p+1$ elements.
Consider $\psi=\psi(k,l)=\left( \begin{smallmatrix}
    k& l \\
    l^{-1}(a-k^2) & -k\\
    \end{smallmatrix} \right)$ for some $k$, $l$.
    Then we have 
    \[1-\varphi-\psi=
    \begin{pmatrix}
    1-k & -1-l \\
    -a-l^{-1}(a-k^2)& 1+k\\
    \end{pmatrix}.
    \]
At least one of the values $1-k$ a $1+k$ is nonzero, hence the rank of the matrix $1-\varphi-\psi$ is either one or two, and the rank is one if and only if $\textup{det}(\zobr)=0$. If the matrix is regular, then $|\Gfg|=1$ and there is only one possible constant $c$ for the pair $(\varphi, \psi)$. Otherwise, we choose $\Gfg=\{\vec{0}, \vec{w}\}$, $\vec{w}\notin \img{\zobr}$, so there are two possible constants $c\in\Gfg$.

The determinant 
\begin{equation*}
   \textup{det}(\zobr)=(1-k)(1+k)-(-1-l)(-a-l^{-1}(a-k^2))
\end{equation*}
is equal to zero if and only if
\begin{equation}\label{det}
    k^2-al^2+(1-2a)l-a=0.
\end{equation}

In Subsection \ref{s4:solutions} we showed that there are precisely $p+1$ pairs $(k, l)\in\Z_p^2$ satisfying the equation. Since $a$ is a non-square, every solution $(k, l)$ satisfies $l\neq0$, and therefore correspond to a matrix from $S_\varphi$. Furthermore, the pairs $(0, 1)$ and $(0, -1)$ do not satisfy the equation \ref{det}, hence all solutions $(k, l)$ correspond to matrices $\psi(k, l)$ with orbits with $p+1$ elements in the conjugation action of $C(\varphi)$ on $S_\varphi$.

The size of the set $\Gfg$, where $\psi=\psi(k, l)$, depends on whether the pair $(k, l)$ satisfies the equation \ref{det}. Also, $|\Gfg|$ does not depend on the choice of a particular representative $\psi$ in its orbit, according to Theorem \ref{t:alg}. Since the equation \ref{det} has precisely $p+1$ solutions, all these solutions correspond to the elements of the same orbit, that is, to the one element of the set of orbit representatives $Y_\varphi$.

We can now find the number of the remaining paramedial quasigroups (i.e., the remaining possible triples $(\varphi, \psi, c)$). There are $(p-1)/2$ possible choices of the matrix $\varphi$ (corresponding to non-squares modulo $p$). For every $\varphi$, there are $p-2$ choices of automorphism $\psi\in Y_\varphi\setminus\{\pm\varphi\}$, and for precisely one $\psi^*\in Y_\varphi$ holds that $|G_{\varphi, \psi^*}|=2$, while for the remaining $p-3$ we have $|\Gfg|=1$. The number of the remaining triples $(\varphi, \psi, c)$, $c\in \Gfg$, is therefore
    \[\frac{p-1}{2}(1\cdot2+(p-3)\cdot1)=\frac{1}{2}p^2-p+\frac{1}{2}.\]
    
Summarized, the case dealing with automorphism $\varphi$ without eigenvalues in $\Z_p$ contributes to $\mathrm{pq}(G)$ by $\frac{p^2-p}{2}+\frac{p^2-p}{2}+(\frac{p^2}{2}-p+\frac{1}{2})=\frac{3}{2}p^2-2p+\frac{1}{2}$.

Summing up all the contributions (see the last column of Table \ref{tab:main}), we obtain that the total number is
%chybny vysledok
%\[\bigl(3p-1\bigr)+\bigl(2p^2-\frac{p}{2}-\frac{7}{2}\bigr)+\bigl(2p-1\bigr)+\bigl(\frac{3}{2}p^2-2p+\frac{1}{2}\bigr)=\frac{7}{2}p^2+\frac{5}{2}p-5.\] 
%spravny sucet:
\[\bigl(3p-1\bigr)+\bigl(\frac{5}{2}p^2-3p-\frac{1}{2}\bigr)+\bigl(2p-1\bigr)+\bigl(\frac{3}{2}p^2-2p+\frac{1}{2}\bigr)=4p^2-2.\] 
\end{proof}

%\smallskip

%\begin{table}[ht]
%\[
%\begin{array}{|l|l|l|l|} \hline
%\varphi & \psi & c & \text{number} \\\hline
%\end{array}
%\]
%\caption{Affine forms of simple paramedial quasigroups of order %$p^2$, up to isomorphism.}
%\label{tab:main}
%\end{table}

\begin{center}

\begin{tabular}{|l|l|L{6cm}|l|}
\hline
$\varphi$   & $\psi$ & $c$ & number\\
\hline
\multirow{5}{5em}{$\begin{pmatrix}
    a& 0  \\
    0 & a\\
  \end{pmatrix}$ $a\neq0$ } & \multirow{2}{5em}{$\begin{pmatrix}
    a& 0  \\
    0 & a\\
  \end{pmatrix}$} & $\begin{pmatrix}
    0\\
    0\\
  \end{pmatrix}$, if $a\neq2^{-1}$& $p-2$ \\\cline{3-4} & & $\begin{pmatrix}
    0\\
    0\\
  \end{pmatrix}$, $\begin{pmatrix}
    1\\
    0\\
  \end{pmatrix}$, if $a=2^{-1}$ & 2 \\
\cline{2-4}
& $\begin{pmatrix}
    -a& 0  \\
    0 & -a\\
  \end{pmatrix}$ & $\begin{pmatrix}
    0\\
    0\\
  \end{pmatrix}$ & $p-1$ \\
\cline{2-4}
& \multirow{2}{5em}{$\begin{pmatrix}
    a& 0  \\
    0 & -a\\
  \end{pmatrix}$} & $\begin{pmatrix}
    0\\
    0\\
  \end{pmatrix}$, if $a\neq2^{-1}$ & $p-2$ \\ \cline{3-4} &&  $\begin{pmatrix}
    0\\
    0\\
  \end{pmatrix}$, $\begin{pmatrix}
    1\\
    0\\
  \end{pmatrix}$, if $a=2^{-1}$ & 2 \\
\hline
\end{tabular}

\begin{tabular}{|l|L{4.8cm}|L{6cm}|l|}
\hline
$\varphi$   & $\psi$ & $c$ & number\\
\hline
\multirow{5}{5em}{$\begin{pmatrix}
    a& 0  \\
    0 & b\\
  \end{pmatrix}$ $0<a<b$} & \multirow{2}{5em}{$\begin{pmatrix}
    a& 0  \\
    0 & b\\
  \end{pmatrix}$} & $\begin{pmatrix}
    0\\
    0\\
  \end{pmatrix}$, if $a, b\neq2^{-1}$ & $\binom{p-2}{2}$ \\\cline{3-4} & &  $\begin{pmatrix}
    0\\
    0\\
  \end{pmatrix}$, $\begin{pmatrix}
    1\\
    1\\
  \end{pmatrix}$, if $a=2^{-1}\vee b=2^{-1}$ & $2(p-2)$ \\
\cline{2-4}
& $\begin{pmatrix}
    -a& 0  \\
    0 & -b\\
  \end{pmatrix}$ & $\begin{pmatrix}
    0\\
    0\\
  \end{pmatrix}$ & $\binom{p-1}{2}$ \\
\cline{2-4}
& \multirow{2}{5em}{$\begin{pmatrix}
    \pm a& 0  \\
    0 & \mp b\\
  \end{pmatrix}$} & $\begin{pmatrix}
    0\\
    0\\
  \end{pmatrix}$, if $a\neq2^{-1}$ or $b\neq2^{-1}$, resp. \newline (depends on the choice of the signs) & %\textcolor{blue}{$\binom{p-2}{2}+p-2$}
  $2\binom{p-2}{2}+p-2$
  \\ \cline{3-4} &&  $\begin{pmatrix}
    0\\
    0\\
  \end{pmatrix}$, $\begin{pmatrix}
    1\\
    1\\
  \end{pmatrix}$, \newline if $a=2^{-1}$ or $b=2^{-1}$, resp. \newline (depends on the choice of the signs) & $2(p-2)$ \\
\hline
\multirow{5}{5.2em}{$\begin{pmatrix}
    a& 0  \\
    0 & -a\\
  \end{pmatrix}$ $0<a<-a$} & \multirow{2}{5em}{$\begin{pmatrix}
    a& 0  \\
    1 & -a\\
  \end{pmatrix}$} & $\begin{pmatrix}
    0\\
    0\\
  \end{pmatrix}$, if $a\neq\pm2^{-1}$ & $\frac{p-3}{2}$\\\cline{3-4} & & $\begin{pmatrix}
    0\\
    0\\
  \end{pmatrix}$, $\begin{pmatrix}
    1\\
    0\\
  \end{pmatrix}$,\newline if $a=2^{-1}$ or $a=-2^{-1}$, resp. & 2 \\
\cline{2-4}
& $\begin{pmatrix}
    -a& 0  \\
    1 & a\\
  \end{pmatrix}$ & $\begin{pmatrix}
    0\\
    0\\
  \end{pmatrix}$ & $\frac{p-1}{2}$ \\
\cline{2-4}
& \multirow{2}{6.7em}{$\begin{pmatrix}
    k& 1  \\
    a^2-k^2 & -k\\
  \end{pmatrix}$} & $\begin{pmatrix}
    0\\
    0\\
  \end{pmatrix}$, if $k\neq2^{-1}a^{-1}-a$& $\frac{(p-1)^2}{2}$ \\ \cline{3-4} && $\begin{pmatrix}
    0\\
    0\\
  \end{pmatrix}$, $\begin{pmatrix}
    0\\
    1\\
  \end{pmatrix}$, if $k=2^{-1}a^{-1}-a$  & $p-1$ \\
\hline
\multirow{3}{5.5em}{$\begin{pmatrix}
    a& 1 \\
    0 & a\\
  \end{pmatrix}$ $a\neq0$} & \multirow{2}{5em}{$\begin{pmatrix}
    a& 1 \\
    0 & a\\
  \end{pmatrix}$} & $\begin{pmatrix}
    0\\
    0\\
  \end{pmatrix}$, if $a\neq2^{-1}$ & $p-2$ \\\cline{3-4} & & $\begin{pmatrix}
    0\\
    0\\
  \end{pmatrix}$, $\begin{pmatrix}
    0\\
    1\\
  \end{pmatrix}$, if $a=2^{-1}$ & 2 \\
\cline{2-4}
& $\begin{pmatrix}
    -a& -1 \\
    0 & -a\\
  \end{pmatrix}$ & $\begin{pmatrix}
    0\\
    0\\
  \end{pmatrix}$ & $p-1$ \\
\hline
\multirow{2}{5.5em}{$\begin{pmatrix}
    0& 1 \\
    a & b\\
  \end{pmatrix}$ $x^2-bx-a$ irreducible} & $\begin{pmatrix}
    0& 1 \\
    a & b\\
  \end{pmatrix}$ & $\begin{pmatrix}
    0\\
    0\\
  \end{pmatrix}$ & $\frac{p^2-p}{2}$\\[12pt]
\cline{2-4}
&$\begin{pmatrix}
    0& -1 \\
    -a & -b\\
  \end{pmatrix}$ & $\begin{pmatrix}
    0\\
    0\\
  \end{pmatrix}$ & $\frac{p^2-p}{2}$\\[12pt]
\hline
\multirow{2}{5.5em}{$\begin{pmatrix}
    0& 1 \\
    a & 0\\
  \end{pmatrix}$ $x^2-a$ irreducible} &
  $p-3$ matrices of the form \newline
  $\mm {k}{l}{l^{-1}(a-k^2)}{-k}$
  \newline $l\neq0$, $(k,l)\neq(0,\pm1)$
  & $\begin{pmatrix}
    0\\
    0\\
  \end{pmatrix}$  & $\frac{(p-1)(p-3)}{2}$\\[15pt]
\cline{2-4}
& 
  1 matrix of the form \newline
  $\mm {k}{l}{l^{-1}(a-k^2)}{-k}$
  \newline arbitrary $(k,l)$ satisfying (\ref{det})
& $\begin{pmatrix}
    0\\
    0\\
  \end{pmatrix}$, $\vec{w}$, $\vec{w}\notin\text{Im}(\zobr)$  & $p-1$\\[15pt]
\hline
\end{tabular}
\captionof{table}{Affine forms of simple paramedial quasigroups of order $p^2$, up to isomorphism ($a,b\in\Z_p$)}
\label{tab:main}

\end{center}

%\smallskip

\section{Simple paramedial quasigroups of prime square order}\label{s5}

In affine quasigroups, quasigroup congruences correspond to subgroups of the underlying group invariant with respect to the defining automorphisms. Formally, it can be stated as follows (and proved fairly easily).

\begin{proposition}[{\cite[Theorems 41 and 42]{KN2}}]\label{p:con}
Let $Q=\Q(G,\varphi,\psi,c)$ be an affine quasigroup and $\alpha$ an equivalence on $Q$. Then $\alpha$ is a congruence of $Q$ if and only if there is a subgroup $N\leq G$ satisfying $\varphi(N)=\psi(N)=N$ such that 
\[ (a,b)\in\alpha\ \Leftrightarrow\ a-b\in N \]
\end{proposition}

A quasigroup $Q$ is called \emph{simple} if it possesses no proper congruences (i.e., every homomorphism from $Q$ is injective or trivial). 

\begin{corollary}
An affine quasigroup $Q=\Q(G,\varphi,\psi,c)$ is simple if and only if the group $G$ contains no proper subgroup $N$ satisfying $\varphi(N)=\psi(N)=N$.
\end{corollary}

In particular, the underlying abelian group must not contain any proper characteristic subgroups, hence, must be elementary abelian.

In the rest of the paper, we classify simple paramedial quasigroups over the group $\Z_p^2$ by sorting out Table \ref{tab:main}. The method is based on the following observation: $Q=\Q(\Z_p^2,\varphi,\psi,c)$ is not simple if and only if there is a proper subgroup $N\leq\Z_p^2$, i.e., a subspace of dimension 1, which is invariant with respect to $\varphi,\psi$. In other words, if $\varphi,\psi$ share a common eigenvector. 

The automorphisms $\varphi$ and $\psi$, where $\varphi$ is a matrix without eigenvalues in $\Z_p$, trivially do not share an eigenvector. For all other choices of $\varphi$ we can easily conclude from Table \ref{tab:main} that the automorphisms $\varphi$ and $\psi$ do not share an eigenvector if and only if $\varphi=\m a00{-a}$ and $\psi=\m k1{a^2-k^2}{-k}$, $k\neq\pm a$. In this case, we need to calculate the number of quasigroups $\Q(G, \varphi, \psi, c)$, where $\varphi$, $\psi$ satisfy the stated condition. 

\smallskip
\begin{center}

\begin{tabular}{|l|L{4.8cm}|L{5.1cm}|l|}
\hline
$\varphi$ & $\psi$ & $c$ & number \\
\hline
\multirow{2}{5.5em}{$\mm a00{-a}$ $0<a<-a$}& \multirow{2}{7em}{$\mm k1{a^2-k^2}{-k}$ $k\neq\pm a$} & $\begin{pmatrix}
    0\\
    0\\
  \end{pmatrix}$, if $k\neq2^{-1}a^{-1}-a$& $\frac{p^2-4p+5}{2}$ \\ \cline{3-4} && $\begin{pmatrix}
    0\\
    0\\
  \end{pmatrix}$, $\begin{pmatrix}
    0\\
    1\\
  \end{pmatrix}$, if $k=2^{-1}a^{-1}-a$  & $p-3$ \\
\hline
\multirow{2}{5.5em}{$\begin{pmatrix}
    0& 1 \\
    a & b\\
  \end{pmatrix}$ $x^2-bx-a$ irreducible} & $\begin{pmatrix}
    0& 1 \\
    a & b\\
  \end{pmatrix}$ & $\begin{pmatrix}
    0\\
    0\\
  \end{pmatrix}$ & $\frac{p^2-p}{2}$\\[12pt]
\cline{2-4}
&$\begin{pmatrix}
    0& -1 \\
    -a & -b\\
  \end{pmatrix}$ & $\begin{pmatrix}
    0\\
    0\\
  \end{pmatrix}$ & $\frac{p^2-p}{2}$\\[12pt]
\hline
\multirow{2}{5.5em}{$\begin{pmatrix}
    0& 1 \\
    a & 0\\
  \end{pmatrix}$ $x^2-a$ irreducible} & 
  $p-3$ matrices of the form \newline
  $\mm {k}{l}{l^{-1}(a-k^2)}{-k}$
  \newline $l\neq0$, $(k,l)\neq(0,\pm1)$
  & $\begin{pmatrix}
    0\\
    0\\
  \end{pmatrix}$  & $\frac{(p-1)(p-3)}{2}$\\[15pt]
\cline{2-4}
&
 1 matrix of the form \newline
  $\mm {k}{l}{l^{-1}(a-k^2)}{-k}$
  \newline arbitrary $(k,l)$ satisfying (\ref{det})
& $\begin{pmatrix}
    0\\
    0\\
  \end{pmatrix}$, $\vec{w}$, $\vec{w}\notin\text{Im}(\zobr)$  & $p-1$\\[15pt]
\hline
\end{tabular}
\captionof{table}{Affine forms of simple paramedial quasigroups of order $p^2$, up to isomorphism.}
\label{tab:simple}
\end{center}

Firstly, we observe that $a=2^{-1}a^{-1}-a$ if and only if $a=\pm2^{-1}$ and for all $a\in\Z_p\setminus\{0\}$ $-a\neq2^{-1}a^{-1}-a$. If $a=\pm2^{-1}$ (only one choice satisfies $0<a<-a$), then $k\neq\pm a$ implies $k\neq2^{-1}a^{-1}-a$, therefore there is only one possible constant $c$ and the number of admissible triples $(\varphi, \psi, c)$ is $p-2$, which is the number of possible choices of $k\neq\pm a$.
For $a\neq\pm2^{-1}$, there are $\frac{p-1}{2}-1$ values of $a$ satisfying $0<a<-a$. Then the number of possible constants $c$ depends on the value of $k$. If $k\neq2^{-1}a^{-1}-a$, there are $p-3$ possible values of $k$ and only one possible constant $c$, hence the number of possible triples $(\varphi, \psi, c)$ is $\left(\frac{p-1}{2}-1\right)\cdot(p-3)$. Otherwise, there are two possible choices of constant $c$ and the number of triples $(\varphi, \psi, c)$ is $\left(\frac{p-1}{2}-1\right)\cdot2$.

%\begin{center}

%\begin{table}[ht]
%\[
%\begin{array}{|l|l|l|l|} \hline
%\varphi & \psi & c & \text{number} \\\hline
%\end{array}
%\]
%\caption{Affine forms of simple paramedial quasigroups of order $p^2$, up to isomorphism.}
%\label{tab:simple}
%\end{table}

%\end{center}

\end{document}